\documentclass[12pt]{article}
\usepackage[top=2.5cm,bottom=2.5cm,left=2.5cm,right=2.5cm]{geometry}
\usepackage{amssymb}
\usepackage{amsmath,amsthm}
\usepackage[latin1]{inputenc}
\usepackage[dvips]{graphicx}
\usepackage{hyperref}
\usepackage{color}
\usepackage{mathrsfs}
\usepackage{enumerate}
\usepackage{tikz}
\usepackage{xifthen}
\usepackage{verbatim}
\usepackage{soul}

\hypersetup{colorlinks=true}

\hypersetup{colorlinks=true, linkcolor=blue, citecolor=blue,urlcolor=blue}

\newcommand{\n}{\operatorname{n}}

\newtheorem{remark}{Remark}[section]

\newtheorem{lemma}[remark]{Lemma}
\newtheorem{theorem}[remark]{Theorem}
\newtheorem{proposition}[remark]{Proposition}

\title{Roman domination in direct product graphs and rooted product graphs}
\author{Abel Cabrera Mart\'{\i}nez$^{(1)}$, Iztok Peterin$^{(2,3)}$ and Ismael G. Yero$^{(4)}$\\
  \\
$^{(1)}${\small Universitat Rovira i Virgili}\\
{\small Departament d'Enginyeria Inform\`atica i Matem\`atiques, Spain.}
\\{\small abel.cabrera\@@urv.cat}\\
$^{(2)}$ {\small University of Maribor}\\
{\small Faculty of Electrical Engineering and Computer Science, Slovenia.}\\
$^{(3)}$ {\small IMFM, Jadranska 19, 1000 Ljubljana, Slovenia.}\\
{\small iztok.peterin\@@um.si} \\
$^{(4)}$ {\small Universidad de C\'adiz}\\
{\small Departamento de Matem\'aticas, Escuela Polit\'ecnica Superior de Algeciras, Spain.}\\
{\small ismael.gonzalez\@@uca.es}
}

\date{ }
\begin{document}
\maketitle

\begin{abstract}
Let $G$ be a graph with vertex set $V(G)$. A function $f:V(G)\rightarrow \{0,1,2\}$ is a Roman dominating function on $G$ if every vertex $v\in V(G)$ for which $f(v)=0$ is adjacent to at least one vertex $u\in V(G)$ such that $f(u)=2$. The Roman domination number of $G$ is the minimum weight $\omega(f)=\sum_{x\in V(G)}f(x)$ among all Roman dominating functions $f$ on $G$. In this article we study the Roman domination number of direct product graphs and rooted product graphs. Specifically, we give several tight lower and upper bounds for the Roman domination number of direct product graphs involving some parameters of the factors, which include the domination, (total) Roman domination, and packing numbers among others. On the other hand, we prove that the Roman domination number of rooted product graphs can attain only three possible values, which depend on the order, the domination, and the Roman domination numbers of the factors in the product. In addition, theoretical characterizations of the classes of rooted product graphs achieving each of these three possible values are given.
\end{abstract}

{\it Keywords}:  Roman domination; domination; Direct product graph; Rooted product graph.

\section{Introduction}

Let $G$ be a simple graph with\emph{ vertex set} $V(G)$ where $\n(G)=|V(G)|$. Given a vertex $v$ of $G$, $N_G(v)$ will denote the \emph{open neighborhood} of $v$ in $G$. The \emph{closed neighborhood}, denoted by $N_G[v]$, equals $N_G(v) \cup \{v\}$. As usual, the graph obtained from $G$ by removing the vertex $v$ (and all the edges incident with it) will be denoted by $G-v$. Given a set $S\subseteq V(G)$, its \emph{open neighborhood} is the set $N_G(S)= \cup_{v\in S} N_G(v)$, and its \emph{closed neighborhood} is the set $N_G[S]= N_G(S)\cup S$. A vertex $v$ is called \emph{universal} if $N_G[v]=V(G)$. By $G[S]$ we denote the subgraph of $G$ induced by $S$.

A set $S\subseteq V(G)$ is a \emph{dominating set} of $G$ if $N_G[S]=V(G)$. The \emph{domination number} of $G$, denoted by $\gamma(G)$, is the minimum cardinality of a dominating set of $G$. A dominating set of $G$ with minimum cardinality is called a $\gamma(G)$-set. More information on domination in graphs can be found in the books \cite{Haynes1998a,Haynes1998}.
In the last two decades, dominating functions have been extensively studied. One of the reasons may be due to the fact that dominating functions generalize the concept of dominating sets. A function $f: V(G) \rightarrow \{0,1,\ldots \}$ on $G$ is said to be a \emph{dominating function} if for every vertex $v$ such that $f(v)=0$, there exists a vertex $u\in N(v)$, such that $f(u)\geq 1$. If we restrict ourselves to the case of functions $f: V(G) \rightarrow \{0,1,2 \}$, then we observe that  $f$ generates three sets $V_0$, $V_1$ and $V_2$; where $V_i=\{v\in V(G) : f(v)=i\}$ for $i\in\{0,1,2\}$. In such a sense, we will write $f(V_0,V_1,V_2)$ to refer to the function $f$. Given a set $S\subseteq V(G)$, $f(S)=\sum_{v\in S}f(v)$. The \emph{weight} of $f$ is $\omega(f)=f(V(G))=|V_1|+2|V_2|$.

The theory of Roman dominating functions is one of the most studied topics within the theory of dominating functions in graphs. Roman dominating functions were formally defined by Cockayne et al. \cite{CDH04} motivated, in part, by a paper of Ian Stewart entitled ``Defend the Roman Empire" \cite{St99}. A \textit{Roman dominating function} (RDF) on a graph $G$ is a function $f(V_0,V_1,V_2)$ such that for every vertex $v\in V_0$, there exists a vertex $u\in N(v)\cap V_2$. The \textit{Roman domination number} of $G$, denoted by $\gamma_R(G)$, is the minimum weight $\omega(f)=\sum_{v\in V(G)}f(v)$  among all  RDFs $f$ on $G$. Some results on Roman domination in graphs can be found for example, in \cite{Chambers2009,CDH04,Favaron2009R,Liu2012a,IPL-1}.

Moreover, one of the most attractive research approaches within the domination theory in graphs is the study of domination-related parameters in product graphs. As expected, some of the researches concerning Roman dominating functions are related to product graphs. In particular, we cite the following works: lexicographic product graphs \cite{CabreraPerfectRomLexi,Roman-lexicographic-2012}; Cartesian and strong product graphs \cite{Yero,YeroJA2013}, direct product of paths and cycles \cite{klobucar2014,klobucar2015}, rooted product graphs \cite{Ismael-rooted-domination}, and corona product graphs \cite{Yero-K-RA}.

It is now our goal to continue with the study of this parameter in two of the product graphs mentioned above. In Section \ref{SectionDirect} we obtain tight bounds for the Roman domination number of direct product graphs. In Section \ref{SectionRooted} we provide closed formulas for the Roman domination number of rooted product graphs, and characterize the graphs reaching these expressions. We end this present section with some terminology on invariants related to domination which will be further used.

A natural lower bound for $\gamma(G)$ is the \emph{packing number} of $G$. A set $A\subseteq V(G)$ is called a \emph{packing set} (or simply a \emph{packing}) if any two distinct vertices $x,y\in A$ satisfy that $N[x]\cap N[y]=\emptyset$, that is, the closed neighborhoods of vertices in $A$ have pairwise empty intersections. The cardinality of a largest packing in $G$ is called the \emph{packing number} and is denoted by $\rho(G)$. Any packing of cardinality $\rho(G)$ is called a $\rho(G)$-set.

Related to packing sets, the notion of an \emph{open packing} in a graph $G$ comes by using open neighborhoods instead of closed neighborhoods. That is, a set $B$ is an \emph{open packing}, if open neighborhoods centered in vertices of $B$ have pairwise empty intersection. The cardinality of a largest open packing in $G$ is called the \emph{open packing number} of $G$ and is denoted by $\rho_o(G)$. An open packing of cardinality $\rho_o(G)$ is simply called a $\rho_o(G)$-set.

The open packing number is a natural lower bound for the \emph{total domination number} $\gamma_t(G)$. This is the minimum cardinality of a set $D\subseteq V(G)$, called a \emph{total dominating set}, where every vertex of $G$ has a neighbor in $D$. A total dominating set of cardinality $\gamma_t(G)$ is called a $\gamma_t(G)$-set as usual.

The last invariant we mention here is the \emph{total Roman domination number} $\gamma_{tR}(G)$. An RDF $f(V_0,V_1,V_2)$ is called a total Roman dominating function (TRDF) if $G[V_1\cup V_2]$ has no isolated vertices. The minimum weight $\omega(f)=|V_1|+2|V_2|$ among all TRDFs $f(V_0,V_1,V_2)$ is called the \emph{total Roman domination number} $\gamma_{tR}(G)$ on $G$. Again, we simply call a TRDF $f$ of weight $\gamma_{tR}(G)$ as a $\gamma_{tR}(G)$-function.

For some extra information (main results, open problems, etc.) on several domination related invariants, including the above mentioned ones, we suggest the recent books \cite{Haynes2020,Haynes2020-a,Henning2013}.

\section{Direct product graphs}\label{SectionDirect}

Let $G$ and $H$ be two graphs. Their direct product $G\times H$ is a graph with vertex set $V(G)\times V(H)$ and two vertices $(u,v)$ and $(u',v')$ are adjacent in $G\times H$ if $uu'\in E(G)$ and $vv'\in E(H)$. The direct product belongs to the so-called four standard graph products, and it is the only one whose edges project to edges to both factors. On the other hand, direct product seems to be quite elusive from many perspectives. Let us mention connectivity, where it can occur that $G\times H$ is disconnected even when both factors $G$ and $H$ are connected. This happens when both $G$ and $H$ are bipartite as shown in \cite{Weic}, see also \cite{HaIK}.

We start with Roman domination in direct products of complete graphs that yields a palette of sharp examples for the bounds that follows.

\begin{proposition}\label{complete}
For integers $r$ and $t$ we have
\begin{equation*}\label{dist}
\gamma_{R}(K_r\times K_t)=\left\{
\begin{array}{ll}
4; & \mbox{if $\;t\geq r=2$},\\
5; & \mbox{if $\;t\geq r=3$},\\
6; & \mbox{if $\;t\geq r>3$}.\\
\end{array}%
\right.
\end{equation*}%
\end{proposition}

\begin{proof}
Let $V(K_r)=\{v_1,\dots,v_r\}$ and $V(K_t)=\{u_1,\dots,u_t\}$. Let first $t\geq r=2$. In this case we have $K_r\times K_t\cong K_{t,t}-M$ for a perfect matching $M=\{(v_1,u_i)(v_2,u_i):i\in\{1,\dots,t\}\}$. Clearly, for any $i\in\{1,\dots,t\}$, the function $f(V_0,V_1,V_2)$, defined by $V_2=\{(v_1,u_i),(v_2,u_i)\}$, $V_1=\emptyset $ and $V_0=V(K_r\times K_t)\setminus V_2$, is an RDF on $K_2\times K_t$ and so, $\gamma_{R}(K_2\times K_t)\leq 4$. Suppose that $\gamma_{R}(K_2\times K_t)< 4$, i.e, $\gamma_{R}(K_2\times K_t)=3$. Let $g(W_0,W_1,W_2)$ be a $\gamma_{R}(K_2\times K_t)$-function. Then there is either one vertex in $W_2$ and one in $W_1$; or three vertices in $W_1$ and no in $W_2$. The last option is not possible since $K_2\times K_t$ contains at least four vertices. Also, the first possibility leads to a contradiction because there are at least two vertices nonadjacent to the unique vertex $(v_i,u_j)$ in $W_2$. Therefore, $\gamma_{R}(K_2\times K_t)=4$.

Let now $t\geq r=3$. We set $V_2=\{(v_1,u_1),(v_2,u_1)\}$, $V_1=\{(v_3,u_1)\}$ and $V_0=V(K_3\times K_t)\setminus (V_1\cup V_2)$. It is easy to see that $f(V_0,V_1,V_2)$ is an RDF on $K_3\times K_t$ and $\gamma_{R}(K_3\times K_t)\leq 5$ follows. Suppose that $\gamma_{R}(K_3\times K_t)<5$ and let $g(W_0,W_1,W_2)$ be a $\gamma_{R}(K_3\times K_t)$-function. Clearly, $W_2\neq \emptyset$ because there are at least nine vertices in $K_3\times K_t$. Every vertex is nonadjacent to exactly $t+1\geq 4$ other vertices, and therefore $|W_2|=2$. For any pair of vertices $(v_i,u_j),(v_k,u_\ell)$ there exists at least one vertex nonadjacent to both. Indeed, if $j=\ell$, then $(v_m,u_j)$ is such for $\{i,k,m\}=\{1,2,3\}$ and if $j\neq\ell$, then $(v_i,u_\ell)$ and $(v_k,u_j)$ are such. In both cases we obtain a contradiction, and the equality $\gamma_{R}(K_3\times K_t)= 5$ holds.

Finally, let $t\geq r>3$. It is easy to check that the function $f(V_0,V_1,V_2)$, defined by $V_2=\{(v_1,u_1),(v_2,u_1),(v_1,u_2)\}$, $V_1=\emptyset $ and $V_0=V(K_r\times K_t)\setminus (V_1\cup V_2)$, is an RDF on $K_r\times K_t$, and we have $\gamma_R(K_r\times K_t)\leq 6$. If $\gamma_R(K_r\times K_t)<6$, then there exists an RDF $g(W_0,W_1,W_2)$ with $\omega(g)\le 5$ and $|W_2|\leq 2$. Clearly, $|W_2|>0$ since $K_r\times K_t$ contains at least $16$ vertices. If $|W_2|=1$, then we obtain a contradiction because there exist at least six vertices in $V(K_r\times K_t)$ being nonadjacent to the unique vertex in $W_2$. If $W_2=\{(v_i,u_j),(v_k,u_\ell)\}$, then there exist at least two vertices nonadjacent to both. More detailed, if $i=k$, then there are at least two more vertices $(v_i,u_m),(v_i,u_p)$ such that $|\{j,\ell,m,p\}|=4$. A symmetrical argument works when $j=\ell$. When $i\neq k$ and $j\neq \ell$, the vertices $(v_i,u_\ell),(v_k,u_j)$ are not adjacent to vertices in $W_2$, a final contradiction. Hence, the equality $\gamma_R(K_r\times K_t)=6$ holds when $t\geq r>3$.
\end{proof}

We next describe several lower and upper bounds for $\gamma_{R}(G\times H)$ in terms of different parameters on the factors of $G$ and $H$.

\begin{theorem}\label{teo-upper-bound-1}
For any graphs $G$ and $H$ with no isolated vertex,
$$\max\{\rho(G)\gamma_{R}(H), \rho(H)\gamma_{R}(G)\}\leq \gamma_{R}(G\times H)\leq \min\{2\gamma(G)\gamma_{tR}(H), 2\gamma(H)\gamma_{tR}(G)\}.$$
\end{theorem}

\begin{proof}
First, we prove the lower bound. Let $f$ be a $\gamma_{R}(G\times H)$-function and $S=\{u_1,\dots,u_{\rho(G)}\}$  a $\rho(G)$-set.
For any $i\in \{1,\dots,\rho(G)\}$, we construct a function $h_i$ on $H$ as follows. For every $v\in V(H)$, let $h_i(v)=\max\{f(u,v)\,:\,u\in N_G[u_i]\}$.
We claim that $h_i$ is an RDF on $H$. Let $v\in V(H)$ such that $h_i(v)=0$. Notice that every vertex $(u,v)\in N_G[u_i]\times \{v\}$ satisfies that $f(u,v)=0$. In particular, for the vertex $(u_i,v)$, there exists $(u_i',v')\in N_{G\times H}(u_i,v)$ such that $f(u_i',v')=2$. So, there exists $v'\in N_H(v)$ with $h_i(v')=2$.
Consequently, $h_i$ is an RDF on $H$, which implies that $\gamma_{R}(H)\le \omega(h_i)\le f(N_G[u_i]\times V(H))$. Thus,
$$\gamma_{R}(G\times H) \ge \sum_{i=1}^{\rho(G)}f(N_G[u_i]\times V(H))\ge \sum_{i=1}^{\rho(G)} \gamma_{R}(H)
                          = \rho(G)\gamma_{R}(H).$$
By the symmetry of $G\times H$, it is also satisfied that $\gamma_{R}(G\times H)\ge \rho(H)\gamma_{R}(G)$, which completes  the proof of the lower bound.

Next, we proceed to the upper bound.
Let $f'(V_0,V_1,V_2)$ be a $\gamma_{tR}(G)$-function and $D$ a $\gamma(H)$-set. Now, we define $W\subseteq V(H)$ as a set of  minimum cardinality among all total dominating sets $W'$ such that $D\subseteq W'$. Notice that $|W|\leq 2|D|$.
We consider a function $g(W_0,\emptyset, W_2)$ on $G\times H$ as follows. If $(x,y)\in (V_2\times W)\cup (V_1\times D)$, then $g(x,y)=2$, and $g(x,y)=0$ otherwise. We claim that $g$ is an RDF on $G\times H$. Let $(u,v)\in W_0$ and distinguish the next two cases.

\vspace{.1cm}

\noindent
Case 1. $u\in V_0$. In this case,  $N_G(u)\cap V_2\neq \emptyset$ and also,  $N_H(v)\cap W\neq \emptyset$ as $W$ is a total dominating set of $H$. Hence, $N_{G\times H}(u,v)\cap W_2\neq \emptyset$.

\vspace{.1cm}

\noindent
Case 2. $u\in V_1\cup V_2$. In this case, $N_G(u)\cap (V_1\cup V_2)\neq \emptyset$ and also,  $N_H(v)\cap D\neq \emptyset$ as $D$ is a dominating set of $H$. Hence, $N_{G\times H}(u,v)\cap W_2\neq \emptyset$.

\vspace{.1cm}

From the previous cases,  we deduce that $g$ is an RDF on $G\times H$, as required. Therefore,
\begin{align*}
  \gamma_{R}(G\times H) & \le \omega(g) \\
                         & = 2|W_2|\\
                         & = 2(|V_2||W|+|V_1||D|)\\
                         & \leq 2(2|V_2||D|+|V_1||D|)\\
                         & = 2|D|(2|V_2|+|V_1|)\\
                         & = 2\gamma(H)\gamma_{tR}(G).
\end{align*}
By the symmetry of $G\times H$, it is also satisfied that $\gamma_{R}(G\times H)\leq 2\gamma(G)\gamma_{tR}(H)$, which completes the proof.
\end{proof}

The upper bound of Theorem \ref{teo-upper-bound-1} is sharp for instance for $K_r\times K_t$, $t\geq r>3$ by Proposition \ref{complete} and since $\gamma(K_r)=\gamma(K_t)=1$ and $\gamma_{tR}(K_r)=\gamma_{tR}(K_t)=3$. We are not aware of any example that is sharp for the lower bound of Theorem \ref{teo-upper-bound-1}. Moreover, in most cases it seems that one could even add a factor two and the bound is still valid. However this is not true in all cases. In particular, $K_2\times C_5\cong C_{10}$ and we have $\gamma_R(K_2\times C_5)=7$ while we get $\max\{\rho(K_2)\gamma_{R}(C_5), \rho(C_5)\gamma_{R}(K_2)\}=4$.

\begin{theorem}\label{up-bounds}
The following statements hold for any graph $H$ with $\gamma_t(H)=\gamma(H)$.
\begin{enumerate}
\item[{\rm (i)}] For any $\gamma_{tR}(G)$-function $f(V_0,V_1,V_2)$,
$$\gamma_{R}(G\times H)\leq 2\gamma(H)(\gamma_{tR}(G)-|V_2|).$$
\item[{\rm (ii)}] If there exists a $\gamma_{tR}(G)$-function  $f(V_0,V_1,V_2)$ such that $V_2$ is a dominating set of $G$, then $$\gamma_{R}(G\times H)\leq 2\gamma(H)(\gamma_{tR}(G)-\gamma(G)).$$
\end{enumerate}
\end{theorem}

\begin{proof}
Let $D$ be a $\gamma_t(H)$-set. So, $D$ is also a $\gamma(H)$-set. From $f$ and $D$, we define $g(W_0,\emptyset,W_2)$ on $G\times H$ as follows: $W_2=(V_1\cup V_2)\times D$ and $W_0=V(G\times H)\setminus W_2$. Observe that $g$ is an RDF on $G\times H$ by the same reasons as in the proof of Theorem \ref{teo-upper-bound-1}. Hence,
\begin{align*}
  \gamma_{R}(G\times H) & \le \omega(g) \\
                         & = 2|W_2|\\
                         & = 2(|V_2||D|+|V_1||D|)\\
                         & = 2|D|(|V_2|+|V_1|)\\
                         & = 2\gamma(H)(\gamma_{tR}(G)-|V_2|),
\end{align*}
which completes the proof of (i).

Finally, (ii) is an immediate consequence of (i).
\end{proof}

Both bounds of Theorem \ref{up-bounds} are sharp. One can observe this for the graph $P_4\times P_4$. Here $\gamma_t(P_4)=2=\gamma(P_4)$ and $\gamma_{tR}(P_4)=4$ where there exists a TRDF $f(V_0,V_1,V_2)$ on $P_4$ with $|V_2|=2$, and by Theorem \ref{up-bounds}, we get $\gamma_{R}(P_4\times P_4)\leq 8$. The equality follows by  \cite[Theorem 5]{klobucar2015}.

We continue with another lower and upper bounds on $\gamma_R(G\times H)$. Before this, we need to give some well-known results.

\begin{theorem}\label{teo-known-bounds}
The following statements hold for any graphs $G$ and $H$ with no isolated vertex.
\begin{enumerate}
\item[{\rm (i)}] {\rm \cite{TRDF-First-2016}} $\gamma_R(G)\leq \gamma_{tR}(G)\leq 3\gamma(G)$.
\item[{\rm (ii)}] {\rm \cite{Chellali2015}} $\gamma_R(G)\geq \gamma_t(G)$.
\item[{\rm (iii)}] {\rm \cite{CabreraDirectTRDF}} $\gamma_{tR}(G\times H)\leq 2\gamma_t(G)\gamma_{t}(H)$.
\item[{\rm (iv)}] {\rm \cite{Rall2005}} $\gamma_{t}(G\times H)\geq \min\{\rho_o(G)\gamma_{t}(H),\rho_o(H)\gamma_{t}(G)\}$.
\end{enumerate}
\end{theorem}

\begin{theorem}\label{easy}
For any graphs $G$ and $H$ with no isolated vertex,
$$\min\{\rho_o(G)\gamma_{t}(H),\rho_o(H)\gamma_{t}(G)\}\leq \gamma_{R}(G\times H)\leq \min\{2\gamma_t(G)\gamma_{t}(H),6\gamma(G)\gamma(H)\}.$$
\end{theorem}

\begin{proof}
By Theorem \ref{teo-known-bounds} (ii) and (iv) we deduce the lower bound. The upper bound $\gamma_{R}(G\times H)\leq 2\gamma_t(G)\gamma_{t}(H)$ holds by Theorem \ref{teo-known-bounds} (i) and (iii).
Finally, to complete the proof of the upper bound we only need to combine Theorems \ref{teo-upper-bound-1} and \ref{teo-known-bounds} (i).
\end{proof}

All the bounds of Theorem \ref{easy} are sharp. The lower bound is sharp for the graph $K_2\times K_2$. By Theorem \ref{easy}, we get $\gamma_{R}(K_2\times K_2)\geq 4$, and the equality holds by Proposition \ref{complete}. For the first upper bound, $\gamma_{R}(G\times H)\leq 2\gamma_t(G)\gamma_{t}(H)$ we observe the family $P_4\times P_{4k}$. Here $\gamma_t(P_{4k})=2k$, and we have  $\gamma_{R}(P_4\times P_{4k})\leq 8k$. The equality follows by \cite[Theorem 5]{klobucar2015}. Finally, the second upper bound $\gamma_{R}(G\times H)\leq 6\gamma(G)\gamma(H)$ is sharp for the family $K_r\times K_t$, $t\geq r>3$, by Proposition \ref{complete}.

We now improve the upper bound $\gamma_{R}(G\times H)\leq 2\gamma_t(G)\gamma_{t}(H)$ from Theorem \ref{easy}. For this, let $G$ be a graph with no isolated vertex and $D$ a $\gamma_t(G)$-set. Let $D'\subseteq D$ be a dominating set of $G$ of minimum cardinality. Notice that $D'$ is not necessarily a $\gamma(G)$-set. For instance, in $C_8=v_1\dots v_8v_1$ we have a $\gamma_t(C_8)$ set $D=\{v_2,v_3,v_6,v_7\}$ and the only dominating set that is a subset of $D$ is $D$ it self. However, $D$ is not a $\gamma (C_8)$-set because $\gamma (C_8)=3$.

The set $D\setminus D'$ is denoted by $K_G(D)$ and is called a \emph{kernel} of $D$. By $k(G)$ we denote the maximum possible cardinality among all kernels $K_G(D)$ from all $\gamma_t(G)$-sets $D$ and their dominating subsets $D'$. For instance, in $P_5=v_1v_2v_3v_4v_5$, there exists a unique $\gamma_t(P_5)$-set $D=\{v_2,v_3,v_4\}$. The kernel of $D$ is then $K_{P_5}(D)=\{v_3\}$ and $k(P_5)=1$.

\begin{theorem}\label{better}
For any graphs $G$ and $H$ with no isolated vertex,
$$\gamma_{R}(G\times H)\leq 2\gamma_t(G)\gamma_{t}(H)-2k(G)k(H).$$
\end{theorem}

\begin{proof}
Let $D_G$ be a $\gamma_t(G)$-set together with $D'_G$ for which $K_G(D_G)$ has maximum cardinality $k(G)$ and $D_H$ be a $\gamma_t(H)$-set together with $D'_H$ for which $K_H(D_H)$ has maximum cardinality $k(H)$. Let $V_2=(D_G\times D_H)\setminus (K_G(D_G)\times K_H(D_H))$ and $V_0=V(G\times H)\setminus V_2$. We will show that $f(V_0,\emptyset,V_2)$ is an RDF on $G\times H$. Notice that every vertex $(u,v)$ has a neighbor in $D_G\times D_H$ because $D_G$ and $D_H$ are total dominating sets of $G$ and $H$, respectively. If $(u,v)$ is adjacent to $(u',v')\in K_G(D_G)\times K_H(D_H)$, then $u$ is adjacent to some $u_0\in D'_G$. Similarly, $v$ is adjacent to some $v_0\in D'_H$. Hence, $(u,v)$ is adjacent to a vertex $(u_0,v_0)$ from $V_2$, as desired. Therefore, $f$ is an RDF on $G\times H$ of weight $\omega(f)=2|V_2|=2\gamma_t(G)\gamma_{t}(H)-2k(G)k(H)$, and the upper bound follows.
\end{proof}

The upper bound of Theorem \ref{better} is sharp because already first upper bound from Theorem \ref{easy} was sharp. But we have a big family of graphs where it is better.

\begin{proposition}\label{better-1}
For any graphs $G$ and $H$ with universal vertices on at least four vertices,
$$\gamma_{R}(G\times H)=6.$$
\end{proposition}

\begin{proof}
Let $u$ and $v$ be universal vertices of $G$ and $H$, respectively. For every $u'\in V(G)\setminus \{u\}$ and every $v'\in V(H)\setminus \{v\}$ we have $D_G=\{u,u'\}$, $D'_G=\{u\}$, $D_H\{v,v'\}$, $D'_H=\{v\}$ and $\gamma_{R}(G\times H)\leq 6$ follows by Theorem \ref{better}. The  bound $\gamma_{R}(G\times H)\geq 6$ follows by the same reasons as the lower bound in Proposition~\ref{complete} for $t\geq r>3$. Therefore, the proof is complete.
\end{proof}

The following result gives a new upper bound on $\gamma_R(G\times H)$.

\begin{theorem}\label{teo-dom-tdom-n}
For any graphs $G$ and $H$ with no isolated vertex,
$$\gamma_{R}(G\times H)\leq \min\{\gamma(G)(\n(H)+\gamma_{t}(H)), \gamma(H)(\n(G)+\gamma_{t}(G))\}.$$
\end{theorem}

\begin{proof}
Let $D$ be a $\gamma(G)$-set and let $W$ be a $\gamma_t(H)$-set. We consider the function $f(V_0,V_1,V_2)$ such that $V_2=D\times W$, $V_1=D\times (V(H)\setminus W)$ and $V_0=V(G\times H)\setminus (V_1\cup V_2)$. It is readily seen that $f$ is an RDF on $G\times H$. Since $\omega(f)=\gamma(G)(\n(H)+\gamma_{t}(H))$, and by using a symmetrical argument, we deduce the upper bound.
\end{proof}

The bound above is tight for instance if we consider the family of direct product graphs $P_3\times K_n$ with $n\ge 3$. By Theorem \ref{teo-dom-tdom-n} we get $\gamma_{R}(P_3\times K_n)\leq 5$. The equality is obtained by the same arguments as in the proof of Proposition \ref{complete} for $\gamma_{R}(K_3\times K_n)$. Other sharp examples are $\gamma_{R}(K_3\times K_n)=5$, by Proposition \ref{complete}, and $\gamma_{R}(P_4\times P_{3k})=6k$, by \cite[Theorem 5]{klobucar2015}.



\section{Rooted product graphs}\label{SectionRooted}

Given a nontrivial graph $G$ and a nontrivial graph $H$ with root $v\in V(H)$, the \emph{rooted product graph} $G\circ_v H$ is defined as the graph obtained from $G$ and $H$ by taking one copy of $G$ and $\n(G)$ copies of $H$ and identifying the  $i^{th}$-vertex of $G$ with the vertex $v$ in the $i^{th}$-copy of $H$ for every $i\in \{1,\dots,\n(G)\}$ \cite{Godsil1978}.
Figure \ref{fig-complex} shows an example of a rooted product graph.

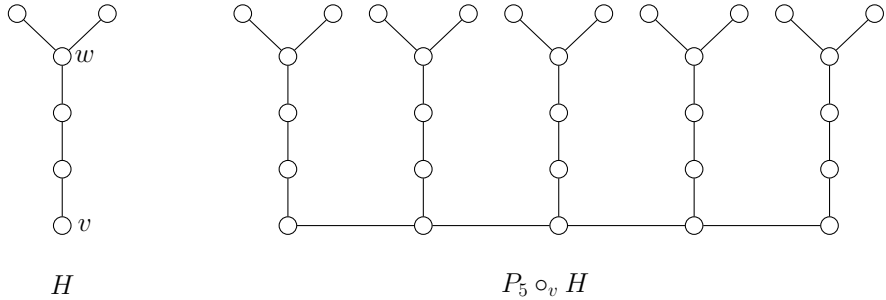
\begin{figure}[ht]
\centering
\begin{tikzpicture}[scale=.6, transform shape]

\node [draw, shape=circle] (aa1) at  (-5,0) {};
\node [draw, shape=circle] (bb1) at  (-5,1.25) {};
\node [draw, shape=circle] (cc1) at  (-5,2.5) {};
\node [draw, shape=circle] (dd1) at  (-5,3.75) {};
\node [draw, shape=circle] (ee2) at  (-4,4.7) {};
\node [draw, shape=circle] (ee3) at  (-6,4.7) {};

\node at (-4.5,0) {\Large $v$};
\node at (-4.5,3.75) {\Large $w$};
\node at (-5,-1.3) {\Large $H$};

\node [draw, shape=circle] (a1) at  (0,0) {};
\node [draw, shape=circle] (a2) at  (3,0) {};
\node [draw, shape=circle] (a3) at  (6,0) {};
\node [draw, shape=circle] (a4) at  (9,0) {};
\node [draw, shape=circle] (a5) at  (12,0) {};
\node at (5.7,-1.3) {\Large $P_5\circ_v H$};

\node [draw, shape=circle] (a11) at  (0,1.25) {};
\node [draw, shape=circle] (a12) at  (0,2.5) {};
\node [draw, shape=circle] (a13) at  (0,3.75) {};
\node [draw, shape=circle] (a14) at  (1,4.7) {};
\node [draw, shape=circle] (a15) at  (-1,4.7) {};

\node [draw, shape=circle] (a21) at  (3,1.25) {};
\node [draw, shape=circle] (a22) at  (3,2.5) {};
\node [draw, shape=circle] (a23) at  (3,3.75) {};
\node [draw, shape=circle] (a24) at  (4,4.7) {};
\node [draw, shape=circle] (a25) at  (2,4.7) {};

\node [draw, shape=circle] (a31) at  (6,1.25) {};
\node [draw, shape=circle] (a32) at  (6,2.5) {};
\node [draw, shape=circle] (a33) at  (6,3.75) {};
\node [draw, shape=circle] (a34) at  (7,4.7) {};
\node [draw, shape=circle] (a35) at  (5,4.7) {};

\node [draw, shape=circle] (a41) at  (9,1.25) {};
\node [draw, shape=circle] (a42) at  (9,2.5) {};
\node [draw, shape=circle] (a43) at  (9,3.75) {};
\node [draw, shape=circle] (a44) at  (10,4.7) {};
\node [draw, shape=circle] (a45) at  (8,4.7) {};

\node [draw, shape=circle] (a51) at  (12,1.25) {};
\node [draw, shape=circle] (a52) at  (12,2.5) {};
\node [draw, shape=circle] (a53) at  (12,3.75) {};
\node [draw, shape=circle] (a54) at  (13,4.7) {};
\node [draw, shape=circle] (a55) at  (11,4.7) {};

\draw(aa1)--(bb1)--(cc1)--(dd1)--(ee3);
\draw(dd1)--(ee2);

\draw(a1)--(a2)--(a3)--(a4)--(a5);

\draw(a1)--(a11)--(a12)--(a13)--(a14);
\draw(a13)--(a15);

\draw(a2)--(a21)--(a22)--(a23)--(a24);
\draw(a23)--(a25);

\draw(a3)--(a31)--(a32)--(a33)--(a34);
\draw(a33)--(a35);

\draw(a4)--(a41)--(a42)--(a43)--(a44);
\draw(a43)--(a45);

\draw(a5)--(a51)--(a52)--(a53)--(a54);
\draw(a53)--(a55);

\end{tikzpicture}
\caption{The graph $P_5\circ_v H$.}
\label{fig-complex}
\end{figure}
For every vertex $x\in V(G)$, $H_x$ will denote the copy of $H$ in $G\circ_v H$ containing $x$. The restriction of any function $f: V(G\circ_v H)\rightarrow \{0,1,2\}$ to the set $V(H_x)$ will be denoted by $f_x$.
Hence, if $f$ is a $\gamma_R(G\circ_v H)$-function, then as $V(G\circ_v H)=\cup_{x\in V(G)}V(H_x)$, we obtain that $$\gamma_R(G\circ_v H)=\omega(f)=\sum_{x\in V(G)}\omega(f_x).$$

\vspace{.1cm}

We recall two results given by Kuziak et al. \cite{Ismael-rooted-domination}, which will be useful later.

\begin{lemma}\label{Lemma H-v}{\em \cite{Ismael-rooted-domination}}
Let $H$ be a nontrivial graph. For any  $v\in V(H)$,
$$\gamma_R(H-v)\geq \gamma_R(H)-1.$$
\end{lemma}

\begin{theorem}\label{teo-Roman-rooted}{\em \cite{Ismael-rooted-domination}}
Let $G$ and $H$ be nontrivial graphs. If $v\in V(H)$ is the root of $H$, then
$$\gamma(G)+\n(G)(\gamma_R(H)-1)\leq \gamma_R(G\circ_v H)\leq \n(G)\gamma_R(H).$$
\end{theorem}

We continue this section with a useful lemma.

\begin{lemma}\label{restriction-H}
Let $f(V_0,V_1,V_2)$ be a $\gamma_R(G\circ_v H)$-function. The following statements hold for any vertex $x\in V(G)$:
\begin{enumerate}
\item[{\rm (i)}] $\omega(f_x)\geq \gamma_R(H)-1$, and
\item[{\rm (ii)}] if $\omega(f_x)=\gamma_R(H)-1$, then $f(x)=0$ and $f(S)\leq 1$ for $S=N_{G\circ_v H}(x)\cap V(H_x)$.
\end{enumerate}
\end{lemma}

\begin{proof}
Let  $x\in V(G)$. First, we observe that every vertex in $V_0\cap (V(H_x)\setminus \{x\})$ is adjacent to some vertex in $V_2\cap V(H_x)$. Thus, the function $g$, defined by  $g(x)=\max\{f(x),1\}$ and $g(u)=f(u)$ whenever $u\in V(H_x)\setminus \{x\}$, is an RDF on $H_x$ and with that also on $H$. Now, if $\omega(f_x)<\gamma_R(H)-1$, then $\omega(g)<\gamma_R(H)$, a contradiction. Hence $\omega(f_x)\geq \gamma_R(H)-1$ and the proof of (i) is completed.

For (ii) assume that $\omega(f_x)=\gamma_R(H)-1$ and let $S=N_{G\circ_v H}(x)\cap V(H_x)$. If $f(x)>0$, then  $f_x$ is an RDF on $H_x$, which implies that $\gamma_R(H)-1=\omega(f_x)\geq \gamma_R(H_x)\geq\gamma_R(H)$, a contradiction. Hence, $f(x)=0$. Now, suppose that $f(S)>1$. If there exists a vertex $y\in N(x)\cap V(H_x)\cap V_2$, then $f_x$ is an RDF on $H_x$ and, as above, we obtain a contradiction. So, $N(x)\cap V(H_x)\cap V_2=\emptyset$. Therefore $|N(x)\cap V(H_x)\cap V_1|\geq 2$, which implies that the function $g'$, defined by $g'(x)=2$, $g'(u)=0$ whenever $u\in N(x)\cap V(H_x)\cap V_1$ and $g'(u)=f(u)$ otherwise, is an RDF on $H_x$ of weight at most $\omega(f_x)=\gamma_R(H)-1$, a contradiction. Therefore, $f(S)\leq 1$, which completes the proof.
\end{proof}

\vspace{.1cm}

From Lemma \ref{restriction-H} (i) we deduce that any $\gamma_R(G\circ_v H)$-function $f$  defines two subsets $\mathcal{A}_f$ and $\mathcal{B}_f$ of $V(G)$ as follows.

\begin{align*}
\mathcal{A}_f=&\{x\in V(G): \omega(f_x)\geq \gamma_R(H)\},\\
\mathcal{B}_f=&\{x\in V(G) : \omega(f_x)=\gamma_R(H)-1\}.
\end{align*}

\begin{lemma}\label{UpperBoundRooted}
Let $f$ be a $\gamma_R(G\circ_v H)$-function. If $\mathcal{B}_f\neq \emptyset$, then the following statements hold:
\begin{enumerate}
\item[{\rm (i)}] $\gamma_R(H-v)=\gamma_R(H)-1$, and
\item[{\rm (ii)}] $\gamma_R(G\circ_v H)\leq \gamma_R(G) + \n(G)(\gamma_R(H)-1)$.
\end{enumerate}
\end{lemma}

\begin{proof}
For (i) let $x\in \mathcal{B}_f$. By Lemma \ref{restriction-H} (ii) we obtain that $f(x)=0$. This implies that the function $f_x$ restricted to $V(H_x)\setminus \{x\}$ is an RDF on $H_x-x$. So $\gamma_R(H-v)=\gamma_R(H_x-x)\leq \omega(f_x)-f(x)=\gamma_R(H)-1$. Lemma \ref{Lemma H-v} leads to $\gamma_R(H-v)=\gamma_R(H)-1$, which completes the proof of (i).

For (ii) observe that from any $\gamma_R(G)$-function and any $\gamma_R(H-v)$-function we can construct an RDF on $G\circ_v H$ of weight $\gamma_R(G)+ \n(G)\gamma_R(H-v)$. By the equality given in (i), it follows that $\gamma_R(G\circ_v H)\leq \gamma_R(G)+ \n(G)\gamma_R(H-v)=\gamma_R(G)+ \n(G)(\gamma_R(H)-1)$, which completes the proof.
\end{proof}

Next, we state the three possible values that $\gamma_R(G\circ_v H)$ can reach. We might recall that graphs $G$ and $H$ with $\gamma_R(G\circ_v H)$ equal to each of these three values were already given in \cite{Ismael-rooted-domination}, although there was not proved that these are the only possible values $\gamma_R(G\circ_v H)$ can reach. In this sense, with the next results we completely settle the problem of the Roman domination in rooted product graphs, already initiated in \cite{Ismael-rooted-domination}.

\begin{theorem}\label{teo-principal-rooted}
Let $G$ be a graph without isolated vertices and let $H$ be a nontrivial graph. If $v\in V(H)$ is the root, then
$$\gamma_R(G\circ_v H)\in\{\n(G)\gamma_R(H), \, \gamma_R(G)+\n(G)(\gamma_R(H)-1), \, \gamma(G)+\n(G)(\gamma_R(H)-1)\}.$$
\end{theorem}

\begin{proof}
Let $f(V_0,V_1,V_2)$ be a $\gamma_R(G\circ_v H)$-function.
Now, we consider the subsets $\mathcal{A}_f, \mathcal{B}_f \subseteq V(G)$ associated to $f$ and differentiate the following cases.

\vspace{0.3cm}

\noindent
Case 1.  $\mathcal{B}_f=\emptyset$.
In this case, for any $x\in V(G)$ it follows that $\omega(f_x)\geq  \gamma_R(H)$. This implies that $\gamma_R(G\circ_v H)=\omega(f)=\sum_{x\in V(G)}\omega(f_x)\geq \n(G)\gamma_R(H)$, and by Theorem \ref{teo-Roman-rooted}, we deduce that
$\gamma_R(G\circ_v H)=\n(G)\gamma_R(H)$.

\vspace{0.3cm}

\noindent
Case 2.  $\mathcal{B}_f\neq \emptyset$. By Theorem \ref{teo-Roman-rooted} we have that $\gamma_R(G\circ_v H)\geq \gamma(G)+\n(G)(\gamma_R(H)-1).$ Now, we assume that $\gamma_R(G\circ_v H)> \gamma(G)+\n(G)(\gamma_R(H)-1)$. Let $z\in \mathcal{B}_f$ and $S=N_{G\circ_v H}[z]\cap V(H_z)$ such that $f(S)$ is maximum among all vertices in $\mathcal{B}_f$. By Lemma \ref{restriction-H} (ii) we obtain that $f(z)=0$ and $f(S)\leq 1$. We analyze the next two subcases.

\vspace{.2cm}

Subcase 2.1. $f(S)=1$. In this subcase there exists a vertex $y\in S\cap V_1$. On the way to a contradiction we define a function $h$ on $H_z$ as follows: $h(z)=2$, $h(y)=0$ and $h(u)=f_z(u)$ whenever $u\in V(H_z)\setminus \{z,y\}$. It is easy to see that $h$ is an RDF on $H_z$ of weight $\omega(h)=\gamma_R(H_z)$, \emph{i.e.}, $h$ is a $\gamma_R(H_z)$-function. Next, from $h$, $f_z$ and any $\gamma(G)$-set $D$, we define a function $f'$ on $G\circ_v H$ as follows. For every $x\in D$, the restriction of $f'$ to $V(H_x)$ is induced from $h$. Moreover, if $x\in V(G)\setminus D$, then the restriction of $f'$ to $V(H_x)$ is induced from $f_z$. Note that $f'$ is an RDF on $G\circ_v H$ of weight $\omega(f')\leq |D|\gamma_R(H_z)+ |V(G)\setminus D|(\gamma_R(H_z)-1)=\gamma(G)+\n(G)(\gamma_R(H)-1)$, which is a contradiction.

\vspace{.2cm}

Subcase 2.2. $f(S)=0$. In this subcase, we first proceed to show that if $y\in \mathcal{A}_f\cap V_2$, then $\omega(f_y)>\gamma_R(H)$. Suppose that there exists a vertex $y'\in \mathcal{A}_f\cap V_2$ such that $\omega(f_{y'})=\gamma_R(H)$. Proceeding analogously as in the Subcase 2.1, from $f_{y'}$, $f_z$ and any $\gamma(G)$-set $D$ we can construct an RDF on $G\circ_v H$ of weight $\gamma(G)+\n(G)(\gamma_R(H)-1)$, which is a contradiction. Hence, for every $y\in \mathcal{A}_f\cap V_2$, it follows that $\omega(f_y)>\gamma_R(H)$, as required.

We now observe that, as $f(S)=0$,  every vertex $x\in \mathcal{B}_f$ satisfies that $f(N_{G\circ_v H}[x]\cap V(H_x))=0$, which implies that $N_{G\circ_v H}(x)\cap \mathcal{A}_f\cap V_2\neq \emptyset$. Thus, $\mathcal{A}_f\cap V_2$ dominates $\mathcal{B}_f$.

Next, we define a function $h$ on $G$ as follows. If $x\in \mathcal{B}_f$, then $h(x)=0$; if $x\in \mathcal{A}_f\cap V_2$, then $h(x)=2$; and finally, if $x\in \mathcal{A}_f\setminus V_2$, then $h(x)=1$. Note that $h$ is an RDF on $G$ of weight $2|\mathcal{A}_f\cap V_2|+|\mathcal{A}_f\setminus V_2|$. Therefore,
\begin{align*}
 \gamma_R(G\circ_v H)&=\sum_{x\in \mathcal{A}_f\cap V_2}\omega(f_x)+\sum_{x\in \mathcal{A}_f\setminus V_2}\omega(f_x)+\sum_{x\in \mathcal{B}_f}\omega(f_x)\\
                       &\geq \sum_{x\in \mathcal{A}_f\cap V_2}(\gamma_R(H)+1)+\sum_{x\in \mathcal{A}_f \setminus V_2}\gamma_R(H)+\sum_{x\in \mathcal{B}_f}(\gamma_R(H)-1)\\
                       &=2|\mathcal{A}_f\cap V_2|+|\mathcal{A}_f \setminus V_2|+\sum_{x\in V(G)}(\gamma_R(H)-1)\\
                       &\geq \gamma_R(G) +\n(G)(\gamma_R(H)-1).
\end{align*}

\noindent
Finally, by Lemma \ref{UpperBoundRooted} (ii) we deduce that $\gamma_{oiR}(G\circ_v H)= \gamma_R(G) +\n(G)(\gamma_R(H)-1)$ and the proof is complete.
\end{proof}

In \cite{Ismael-rooted-domination}, the authors provided some sufficient conditions under which the three possible values of Roman domination number of rooted product graphs are achieved. In addition, we next give two examples in which we can observe that these expressions of $\gamma_R(G\circ_v H)$ are realizable.

\begin{itemize}
\item If $G$ is a graph with no isolated vertex and $H$ is the graph shown in Figure \ref{fig-complex}, then
\begin{itemize}
\item $\gamma_R(G\circ_v H)=\gamma(G)+\n(G)(\gamma_R(H)-1)=\gamma(G)+3\n(G)$.
\item $\gamma_R(G\circ_w H)=\n(G)\gamma_R(H)=4\n(G)$.
\end{itemize}
\item If $G$ is a graph with no isolated vertex and $H$ is the path $P_4$ (where the root $v\in V(P_4)$ is a vertex of degree one), then  $\gamma_R(G\circ_v H)=\gamma_R(G)+\n(G)(\gamma_R(H)-1)=\gamma_R(G)+2\n(G)$.
\end{itemize}

We next proceed to characterize the graphs $G\circ_v H$ with $\gamma_R(G\circ_v H)=\gamma(G)+ \n(G)(\gamma_R(H)-1)$.

\begin{theorem}\label{teo-char-1}
Let $G$ be a graph with no isolated vertex and $H$ any nontrivial graph with root $v\in V(H)$. The following statements are equivalent.
\begin{enumerate}
\item[{\rm (i)}] $\gamma_R(G\circ_v H)=\gamma(G)+ \n(G)(\gamma_R(H)-1)$.
\item[{\rm (ii)}] $\gamma_R(H-v)=\gamma_R(H)-1$ and there exists a $\gamma_R(H)$-function $g$ such that $g(v)=2$.
\end{enumerate}
\end{theorem}

\begin{proof}
First, we assume that (i) holds, \emph{i.e.}, $\gamma_R(G\circ_v H)=\gamma(G)+ \n(G)(\gamma_R(H)-1)$. Let $f(V_0,V_1,V_2)$ be a $\gamma_R(G\circ_v H)$-function. Since $\gamma(G)<\n(G)$, it follows that $\mathcal{B}_f\neq \emptyset$. So, (i) of Lemma \ref{UpperBoundRooted} leads to $\gamma_R(H-v)=\gamma_R(H)-1$.
By (ii) of Lemma \ref{restriction-H}  we deduce $\mathcal{B}_f\subseteq V_0$, and also, if $x\in \mathcal{B}_f$, then as $f(N_{G\circ_v H}(x)\cap V(H_x))\leq 1$, there exists a vertex $y\in N_{G\circ_v H}(x)\cap \mathcal{A}_f\cap V_2$. Thus, $\mathcal{A}_f$ is a dominating set of $G$. Hence,
\begin{align*}
\gamma(G)+ \n(G)(\gamma_R(H)-1)&=\omega(f)\\
                             &=\sum_{x\in \mathcal{A}_f}\omega(f_x)+\sum_{x\in \mathcal{B}_f}\omega(f_x)\\
                             &\geq\sum_{x\in \mathcal{A}_f}\gamma_R(H)+\sum_{x\in \mathcal{B}_f}(\gamma_R(H)-1)\\
                             &\geq |\mathcal{A}_f|+\n(G)(\gamma_R(H)-1)\\
                             &\geq \gamma(G)+\n(G)(\gamma_R(H)-1)
\end{align*}
So, we have equalities in the inequality chain above, and as a consequence, we deduce that every vertex $y\in \mathcal{A}_f\cap V_2$ satisfies that $\omega(f_x)=\gamma_R(H)$. Thus, as $H_x\cong H$, there exists a $\gamma_R(H)$-function $g$ such that $g(v)=2$.

On the other hand, we assume that $\gamma_R(H-v)=\gamma_R(H)-1$ and that there exists a $\gamma_R(H)$-function $g$ such that $g(v)=2$.
 Let $D$ be a $\gamma(G)$-set and $g'$ be a $\gamma_R(H-v)$-function. From $D$, $g'$ and $g$, we define a function $h$ on $G\circ_v H$ as follows. For every $x\in D$, the restriction of $h$ to $V(H_x)$ is induced by $g$. Moreover, if $x\in V(G)\setminus D$, then $h(x)=0$ and the restriction of $h$ to $V(H_x-x)$ is induced by $g'$. Observe that $h$ is an RDF on $G\circ_v H$, and so $\gamma_R(G\circ_v H)\leq \omega(h)=|D|\gamma_R(H)+|V(G)\setminus D|(\gamma_R(H)-1)=\gamma(G)+ \n(G)(\gamma_R(H)-1)$. Therefore, Theorem \ref{teo-principal-rooted} leads to $\gamma_R(G\circ_v H)=\gamma(G)+ \n(G)(\gamma_R(H)-1)$, which completes the proof.
\end{proof}

Now, we characterize the graphs $G$ and $H$ (and the root $v\in V(H)$) that satisfy the equality $\gamma_R(G\circ_v H)=\gamma_R(G)+ \n(G)(\gamma_R(H)-1)$.

\begin{theorem}\label{teo-char-2}
Let $G\not\cong \cup K_2$ be a graph with no isolated vertex and $H$ any nontrivial graph with root $v\in V(H)$. The following statements are equivalent.
\begin{enumerate}
\item[{\rm (i)}] $\gamma_R(G\circ_v H)=\gamma_R(G)+ \n(G)(\gamma_R(H)-1)$.
\item[{\rm (ii)}] $\gamma_R(H-v)=\gamma_R(H)-1$ and $g(v)\leq 1$ for every $\gamma_R(H)$-function $g$.
\end{enumerate}
\end{theorem}

\begin{proof}
First, we assume that (i) holds, \emph{i.e.}, $\gamma_R(G\circ_v H)=\gamma_R(G)+ \n(G)(\gamma_R(H)-1)$. Let $f(V_0,V_1,V_2)$ be a $\gamma_R(G\circ_v H)$-function. Since $\gamma_R(G)<\n(G)$, it follows that $\mathcal{B}_f\neq \emptyset$. So, Lemma \ref{UpperBoundRooted} (i) leads to $\gamma_R(H-v)=\gamma_R(H)-1$.
Moreover, since $\gamma(G)<\gamma_R(G)$, Theorem \ref{teo-char-1} leads to $g(v)\leq 1$ for every $\gamma_R(H)$-function $g$. Hence, (ii) holds.

On the other hand, assume that (ii) holds. As in the previous proof, from any $\gamma_R(G)$-function and any $\gamma_R(H-v)$-function we can construct an RDF on $G\circ_v H$ of weight $\gamma_R(G)+ \n(G)\gamma_R(H-v)$. Since $\gamma_R(H-v)=\gamma_R(H)-1$, it follows that $\gamma_R(G\circ_v H)\leq \gamma_R(G)+ \n(G)\gamma_R(H-v)=\gamma_R(G)+ \n(G)(\gamma_R(H)-1)$. Hence, Theorem \ref{teo-principal-rooted} leads to $\gamma_R(G\circ_v H)\in \{\gamma_R(G)+ \n(G)(\gamma_R(H)-1), \gamma(G)+ \n(G)(\gamma_R(H)-1)\}$. Finally, as $g(v)\leq 1$ for every $\gamma_R(H)$-function $g$, by Theorem \ref{teo-char-1} we deduce that $\gamma_R(G\circ_v H)=\gamma_R(G)+ \n(G)(\gamma_R(H)-1)$, which completes the proof.
\end{proof}

Finally, we characterize the graphs (and the root) reaching the equality $\gamma_R(G\circ_v H)=\n(G)\gamma_R(H)$. Observe that, as shown in Theorem \ref{teo-principal-rooted}, there are three possible expressions  for the Roman domination number of $G\circ_v H$. Thus, for the case of the equality   $\gamma_R(G\circ_v H)=\n(G)\gamma_R(H)$, the corresponding characterization can be derived by eliminating the previous results (Theorems \ref{teo-char-1} and \ref{teo-char-2}) from the family of all graphs $G$ without isolated vertices and all nontrivial graphs $H$ with root $v\in V(H)$.

\begin{theorem}\label{teo-char-3}
Let $G\not\cong \cup K_2$ be a graph with no isolated vertex and $H$ any nontrivial graph with root $v\in V(H)$. Then $\gamma_R(G\circ_v H)=\n(G)\gamma_R(H)$ if and only if $\gamma_R(H-v)\geq\gamma_R(H)$.
\end{theorem}

\begin{proof}
First, we assume that $\gamma_R(G\circ_v H)=\n(G)\gamma_R(H)$. By Lemma \ref{Lemma H-v} we have that $\gamma_R(H-v)\geq\gamma_R(H)-1$. If $\gamma_R(H-v)=\gamma_R(H)-1$, then by Theorems \ref{teo-char-1} and \ref{teo-char-2} we deduce that $\n(G)\gamma_R(H)=\gamma_R(G\circ_v H)\leq \gamma_R(G)+\n(G)(\gamma_R(H)-1)$, which is a contradiction as $\gamma_R(G)<\n(G)$. Therefore, $\gamma_R(H-v)\geq\gamma_R(H)$, as desired.

On the other hand, we assume that $\gamma_R(H-v)\geq\gamma_R(H)$.
Hence, Theorems \ref{teo-char-1} and \ref{teo-char-2} lead to $\gamma_R(G\circ_v H)\notin \{\gamma(G)+\n(G)(\gamma_R(H)-1), \gamma_R(G)+\n(G)(\gamma_R(H)-1)\}$, which implies that $\gamma_R(G\circ_v H)=\n(G)\gamma_R(H)$ by Theorem \ref{teo-principal-rooted}. Therefore, the proof is complete.
\end{proof}

Clearly, one might recall that the characterizations above rely on the following fact. It is necessary to know the relationship between $\gamma_R(H-v)$ and $\gamma_R(H)$ in the graphs $H$ used in the rooted product, as well as, the existence of $\gamma_R(H)$-functions with labels $1$ or $2$ in their roots. Thus, an interesting open problem that it is then of interest is as follows.\\

\noindent
\textbf{Open question:} For a given graph $H$ and a vertex $v\in V(H)$, which is the relationship that exists between $\gamma_R(H-v)$ and $\gamma_R(H)$?\\

On the other hand, as mentioned before, studies on Roman domination in Cartesian product graphs were published in \cite{Yero,YeroJA2013}; and studies on the rooted product graphs were initiated in \cite{Ismael-rooted-domination}, and continued here in our work. In consequence, it would be of interest to continue these studies by considering the hierarchical product of graphs (see \cite{Barriere}), which is a subgraph of the Cartesian product as well as supergraph of the rooted product.

\section*{Acknowledgments}

The second author (Iztok Peterin) has been partially supported by the Slovenian Research Agency by the projects No. J1-1693 and J1-9109. 
The last author (Ismael G. Yero) has been partially supported by ``Junta de Andaluc\'ia'', FEDER-UPO Research and Development Call, reference number UPO-1263769.


\begin{thebibliography}{}

\bibitem{TRDF-First-2016}
H.~Abdollahzadeh~Ahangar, M.~A. Henning, V.~Samodivkin, I.~G. Yero, Total Roman  domination in graphs, Appl. Anal. Discrete Math. 10 (2016) 501--517.

\bibitem{Barriere}
L. Barriere, F. Comellas, C. Dalf\'o, M. A. Fiol, The hierarchical product of graphs, Discrete App. Math. 157(1) (2009) 36--48.

\bibitem{CabreraPerfectRomLexi}
A. Cabrera Mart\'inez, C. Garc\'ia-G\'omez, J.A. Rodr\'iguez-Vel\'azquez, Perfect domination, Roman domination and perfect Roman domination in lexicographic product graphs, Manuscript (2020).

\bibitem{CabreraDirectTRDF}
A. Cabrera Mart\'inez, D. Kuziak, I. Peterin, I.G. Yero,  Dominating the direct product of two graphs through total Roman strategies, Mathematics 8 (2020) 1438.


\bibitem{Chambers2009}
E.W. Chambers, B. Kinnersley, N. Prince, D.B. West, Extremal problems for Roman domination, SIAM J. Discrete Math. 23 (2009)  1575--1586.

\bibitem{Chellali2015}
 M. Chellali, T. Haynes, S. T. Hedetniemi,  Roman and total domination, Quaest. Math. 38 (2015) 749--757.


\bibitem{CDH04}
E.J. Cockayne, P.A. Jr. Dreyer, S.M. Hedetniemi, S.T. Hedetniemi, Roman domination in graphs, Discret. Math. 278 (2004) 11--22.


\bibitem{Favaron2009R}
O. Favaron, H. Karami, R. Khoeilar, S.M. Sheikholeslami, On the Roman domination number of a graph, Discrete Math. 309 (2009)  3447--3451.

\bibitem{Godsil1978}
C.D. Godsil, B.D. McKay, A new graph product and its spectrum, B. Aust. Math. Soc. 18 (1978) 21--28.

\bibitem{HaIK} R. Hamack, W. Imrich, S. Klav\v{z}ar, Handbook of Product
Graphs, Second Edition, CRC Press, Boca Raton, FL, 2011.

\bibitem{Haynes2020}
T.W. Haynes, S.T. Hedetniemi, M.A. Henning (Eds.), Topics in Domination in Graphs. Developments in Mathematics, Springer International Publishing AG, 2020.

\bibitem{Haynes2020-a}
T.W. Haynes, S.T. Hedetniemi, M.A. Henning (Eds.), Structures of Domination in Graphs. Developments in Mathematics, Springer International Publishing AG, 2020.

\bibitem{Haynes1998a} T. Haynes, S. Hedetniemi, P. Slater, Domination in Graphs: Volume 2: Advanced Topics, Chapman \& Hall/CRC Pure and Applied Mathematics, Taylor \& Francis, 1998.

\bibitem{Haynes1998} T. W. Haynes, S. T. Hedetniemi, P. J. Slater, Fundamentals of Domination in Graphs, Chapman and Hall/CRC Pure and Applied Mathematics Series, Marcel Dekker, Inc. New York, 1998.

\bibitem{Henning2013}
M. A. Henning, A. Yeo, Total domination in graphs. New York, Springer, 2013.

\bibitem{klobucar2014}
A. Klobu\v car, I. Pulji\'c, Some results for Roman domination number on cardinal product of paths and cycles, Kragujev. J. Math. 38 (2014) 83--94.

\bibitem{klobucar2015}
A. Klobu\v car, I. Pulji\'c, Roman domination number on cardinal product of paths and cycles, Croat. Oper. Res. Rev. 6 (2015) 71--78.

\bibitem{Roman-lexicographic-2012}
T. Kraner \v Sumenjak, P. Pavli\v c, A. Tepeh, On the Roman domination in the lexicographic product of graphs, Discrete Appl. Math.  160~(13) (2012) 2030--2036.

\bibitem{Ismael-rooted-domination}
D. Kuziak, M. Lemanska,  I.G. Yero, Domination-related parameters in rooted product graphs, Bull. Malays. Math. Sci. Soc. 39 (2019)  199--217.

\bibitem{Liu2012a}
C.-H Liu, G.J. Chang, Upper bounds on Roman domination numbers of graphs, Discrete Math. 312 (2012) 1386--1391.

\bibitem{Rall2005}
D.F. Rall, Total domination in categorical products of graphs, Discuss. Math. Graph Theory 25 (2005) 35--44.

\bibitem{St99}
I. Stewart, Defend the Roman Empire!, Sci. Am. 281 (1999) 136--138.

\bibitem{Weic} P. M.~Weichsel, The Kronecker product of graphs, Proc. Amer. Math. Soc. 13 (1962) 47--52.

\bibitem{Yero}
I.G. Yero, On Clark \& Suen bound type results for $k$-domination and Roman domination of Cartesian product graphs, Int. J. Comput. Math. 90(3) (2013) 522--526.

\bibitem{Yero-K-RA}
I.G. Yero, D. Kuziak, A. Rond\'on Aguilar, Coloring, location and domination of corona graphs, Aequationes Math. 86 (2013) 1--21.

\bibitem{YeroJA2013}
I.G. Yero, J.A. Rodr\'{\i}guez-Vel\'azquez, Roman domination in Cartesian product graphs and strong product graphs, Appl. Anal. Discr. Math. 7 (2013) 262--274.

\bibitem{IPL-1}
E. Zhu, Z. Shao, Extremal problems on weak Roman domination number, Inf. Process. Lett. 138 (2018) 12--18.


\end{thebibliography}
\end{document}